\documentclass[10pt, english]{article}

\usepackage[margin= 1.5 cm]{geometry}

\usepackage{amsthm}
\usepackage{amsmath}
\usepackage{amssymb}
\usepackage{setspace}
\usepackage{mathtools}
\usepackage{verbatim}
\usepackage{csquotes}
\usepackage{graphicx}
\usepackage[hidelinks]{hyperref}
\usepackage{thm-restate}
\usepackage{cleveref}
\usepackage{graphicx}
\usepackage{appendix}
\usepackage[inline]{enumitem}
\usepackage{framed}
\usepackage{subcaption}

\usepackage{bbm}
\usepackage{mathrsfs}
\usepackage[makeroom]{cancel}
\usepackage{caption}

\usepackage{floatrow}
\usepackage[T1]{fontenc}

\usepackage{tikz}
\usepackage{mathdots}
\usepackage{xcolor}
\usepackage{diagbox}
\usepackage{colortbl}
\usepackage[absolute,overlay]{textpos}

\graphicspath{ {./images/} }

\usetikzlibrary{calc}
\usetikzlibrary{decorations.pathreplacing}
\usetikzlibrary{positioning,patterns}
\usetikzlibrary{arrows,shapes,positioning}
\usetikzlibrary{decorations.markings}

\tikzstyle{edge}=[very thick]
\definecolor{bostonuniversityred}{rgb}{0.8, 0.0, 0.0}
\definecolor{arsenic}{rgb}{0.23, 0.27, 0.29}
\tikzstyle{diredge}=[postaction={decorate,decoration={markings,
		mark=at position .95 with {\arrow[scale = 1]{stealth};}}}]

\newcommand{\defPt}[3]{
	\def \pt {(#1, #2)}
	\coordinate [at = \pt, name = #3];
}

\tikzset{
   conn/.pic={
     \defPt{0.2}{-0.5}{q0}
     \defPt{-1}{-1.5}{q5}
    \defPt{1}{1.2}{q1}
    \defPt{1}{2.7}{q6}
    \defPt{1.25}{-1.2}{q2}
    \defPt{2.5}{0.6}{q3}
    \defPt{2.5}{-0.6}{q4}
  
        \draw[line width=1 pt] (q0) -- (q1) -- (q3) -- (q4);
        \draw[line width=1 pt] (q2) -- (q3);
        \draw[line width=1 pt] (q0) -- (q5);
        \draw[line width=1 pt] (q1) -- (q6);
  }
}

\newcommand{\fitellipsis}[3] 
{\draw []let \p1=(#1), \p2=(#2), \n1={atan2(\y2-\y1,\x2-\x1)}, \n2={veclen(\y2-\y1,\x2-\x1)}
    in ($ (\p1)!0.5!(\p2) $) ellipse [ x radius=\n2/2+0.3cm+#3cm, y radius=#3cm, rotate=\n1];
}

\newcommand{\fitellipsiss}[3] 
{\draw [fill=white]let \p1=(#1), \p2=(#2), \n1={atan2(\y2-\y1,\x2-\x1)}, \n2={veclen(\y2-\y1,\x2-\x1)}
    in ($ (\p1)!0.5!(\p2) $) ellipse [ x radius=\n2/2+#3cm, y radius=#3cm, rotate=\n1];
}

\newcommand{\fitellipsisss}[3] 
{\draw []let \p1=(#1), \p2=(#2), \n1={atan2(\y2-\y1,\x2-\x1)}, \n2={veclen(\y2-\y1,\x2-\x1)}
    in ($ (\p1)!0.5!(\p2) $) ellipse [ x radius=\n2/2+#3cm, y radius=#3cm, rotate=\n1];
}

\theoremstyle{plain}

\newtheorem*{thm*}{Theorem}
\newtheorem{thm}{Theorem}
\Crefname{thm}{Theorem}{Theorems}
\numberwithin{thm}{section}

\newtheorem*{lem*}{Lemma}
\newtheorem{lem}[thm]{Lemma}
\Crefname{lem}{Lemma}{Lemmas}

\newtheorem*{claim*}{Claim}

\Crefname{claim}{Claim}{Claims}
\Crefname{claim}{Claim}{Claims}

\newtheorem{prop}[thm]{Proposition}
\Crefname{prop}{Proposition}{Propositions}

\newtheorem{cor}[thm]{Corollary}
\Crefname{cor}{Corollary}{Corollaries}

\Crefname{conj}{Conjecture}{Conjectures}

\Crefname{qn}{Question}{Questions}

\Crefname{obs}{Observation}{Observations}

\newtheorem{ex}[thm]{Example}
\Crefname{ex}{Example}{Examples}

\theoremstyle{definition}

\Crefname{prob}{Problem}{Problems}

\Crefname{defn}{Definition}{Definitions}

\newtheorem*{defn*}{Definition}

\theoremstyle{remark}

\renewenvironment{proof}[1][]{\begin{trivlist}
\item[\hspace{\labelsep}{\bf\noindent Proof#1.\/}] }{\qed\end{trivlist}}

\newcommand{\floor}[1]{
    \left\lfloor #1 \right\rfloor
}

\expandafter\def\expandafter\normalsize\expandafter{%
    \normalsize
    \setlength\abovedisplayskip{8pt}
    \setlength\belowdisplayskip{8pt}
    \setlength\abovedisplayshortskip{4pt}
    \setlength\belowdisplayshortskip{4pt}
}

\usepackage[square,sort,comma,numbers]{natbib}
 \setlist[itemize]{leftmargin=*}

\newcommand{\G}{\mathcal{G}}

\DeclareFontFamily{OT1}{pzc}{}
\DeclareFontShape{OT1}{pzc}{m}{it}{<-> s * [1.10] pzcmi7t}{}
\DeclareMathAlphabet{\mathpzc}{OT1}{pzc}{m}{it}

\DeclareMathOperator{\SL}{SL}
\DeclareMathOperator{\Cay}{Cay}
\newcommand{\gall}[1]{\G^{\text{all}}_{#1}}
\newcommand{\gpla}[1]{\G^{\text{planar}}_{#1}}
\newcommand{\greg}[1]{\G^{k\text{-regular}}_{#1}}

\newcount\colveccount
\newcommand*\colvec[1]{
        \global\colveccount#1
        \begin{bmatrix}
        \colvecnext
}
\def\colvecnext#1{
        #1
        \global\advance\colveccount-1
        \ifnum\colveccount>0
                \\
                \expandafter\colvecnext
        \else
                \end{bmatrix}
        \fi
}

\title{\vspace{-0.8cm} 
The spanning tree spectrum: improved bounds and simple proofs}

\author{Noga Alon\thanks{Department of Mathematics, Princeton
University, Princeton, USA. Email: \href{mailto:nalon@math.princeton.edu} {\nolinkurl{nalon@math.princeton.edu}}. Research supported in part by NSF grant
DMS-2154082.}
\and
Matija Buci\'c\thanks{Department of Mathematics, Princeton University, Princeton, USA. Email: \href{mailto:mb5225@princeton.edu} {\nolinkurl{mb5225@princeton.edu}}. Research supported in part by an NSF Grant DMS--2349013.}
 \and Lior Gishboliner\thanks{Department of Mathematics, University of Toronto, Canada. Email: \href{mailto:lior.gishboliner@utoronto.ca}{\tt lior.gishboliner@utoronto.ca}. Research supported by the NSERC Discovery Grant ``Problems in Extremal and Probabilistic Combinatorics".}
}

 \date{}

\begin{document}

\maketitle

\vspace{-0.5cm}
\begin{abstract}
The number of spanning trees of a graph $G$, denoted $\tau(G)$, is 
a well studied graph parameter with numerous connections to other areas of mathematics. In a recent remarkable paper, answering a question of Sedl\'a\v{c}ek from 1969, Chan, Kontorovich and Pak showed that $\tau(G)$ takes at least $1.1103^n$ different values across simple (and planar) $n$-vertex graphs $G$, for large enough $n$. We give a very short, purely combinatorial proof that at least $1.55^n$ values are attained. 
We also prove that exponential growth can be achieved with regular graphs, determining the growth rate in another problem first raised by Sedl\'a\v{c}ek in the late 1960's. We further show that the following modular dual version of the result holds.
For any integer $N$ and any $u < N$ there exists a planar graph 
on $O(\log N)$ vertices whose number of spanning trees is $u$
modulo $N$.

\end{abstract} 

\vspace{-0.5
cm}

\section{Introduction} 
\vspace{-0.3
cm}
Given a graph $G$, we denote by $\tau(G)$ the number of spanning trees of $G$. This simple quantity has numerous interpretations, with perhaps 
the simplest and most classical being its expression 
as the determinant of a certain type of symmetric matrices via Kirchhoff's matrix tree theorem \cite{kirchhoff} from 1847. The notion has a long history dating back to 1860 and the first proof of Cayley's formula \cite{cayley}. It can also be viewed as capturing a certain notion of ``complexity'' of $G$. Besides being extensively studied in combinatorics, it has many, often surprising, connections to other areas of mathematics and beyond, including Commutative algebra, Probability theory, Theory of Lie groups, Combinatorial 
Optimization and Electrical Networks.

In the present short paper we are interested in one of 
the most basic questions that can be asked about the function 
$\tau$, namely how large is its range, when evaluated on, say, 
all $n$ vertex graphs? This question was first raised by Sedl\'a\v{c}ek in the 1960's \cite{sed66,sed69,sed70} and has since been reiterated many times, the earliest of which being the book \cite{moon} of Moon from 1970. In fact, he raised this question for three natural families of graphs: $\gall{n},\: \gpla{n},\: \greg{n}$ -- the families of all, planar, and $k$-regular graphs on $n$ vertices, respectively. With this in mind, given a family of simple\footnote{In this paper, all graphs we consider are assumed to be simple, unless otherwise specified.} graphs $\mathcal{G},$ let us define its \emph{tree spectrum} $T(\mathcal{G}):=\{\tau(G) \mid G \in \mathcal{G}\}.$ It is immediate that $|T(\gall{n})|\ge |T(\gpla{n})|,$ and somewhat remarkably, historically, the essentially best known lower bounds on $|T(\gall{n})|$ 
came from lower bounds on $|T(\gpla{n})|,$ despite a seemingly vast amount of additional freedom afforded to graphs in $\gall{n}$ compared to $\gpla{n}$. In particular, for any $G\in \gpla{n},$ by exploiting degeneracy, it is easy to see that $\tau(G) \le 6^n$, and the current best upper bound is $\tau(G)\le 5.2852^n$ \cite{BS10} (compared to $n^{n-2}$ for $G \in \gall{n}$). This also gives the essentially best known upper bound on $|T(\gpla{n})|$.

The first lower bound on $|T(\gpla{n})|$ was given already in one of the original papers by Sedl\'a\v{c}ek in 1969 \cite{sed69}, who proved that $|T(\gpla{n})| = \Omega(n^2)$. In \cite{aza14}, Azarija obtained a major improvement, by showing that $|T(\gpla{n})| \ge 
2^{\Omega(\sqrt{n/\log n})}$. Subsequently, a result of Stong \cite{stong}, who was interested in a certain dual question which we will discuss later, implies that $|T(\gpla{n})| \ge  2^{\Omega(n^{2/3})}$. 
In a remarkable recent paper, Chan, Kontorovich and Pak \cite{CKP} showed that $|T(\gpla{n})| \ge 1.1103^n$, for large $n$. Their argument uses combinatorial ideas combined with ideas 
concerning the theory of continued fractions. 

We give a purely combinatorial, short and simple proof with a significantly stronger exponential bound. 

\begin{thm}\label{thm:planar}
    $|T(\gall{n})|\ge |T(\gpla{n})| \ge 0.9 c^n$ for large enough $n$, where $c = \sqrt{\sqrt{2}+1} \geq 1.55$.
\end{thm}
In fact, we give an even simpler, one page proof of a slightly weaker bound $|T(\gpla{n})| \ge 2^{n/2-1},$ which holds for all $n$. 

Turning to regular graphs, we note that the natural question of determining $|\greg{n}|$ has also been originally raised by Sedl\'a\v{c}ek \cite{sed69} in 1969 and has since been reiterated several times over the years by various authors \cite{moon,sed70b,sed77,aza14,CKP}. In the original paper, Sedl\'a\v{c}ek proved for $k=3$ that $|\greg{n}|$ grows at least linearly in $n$ (provided $n$ is even\footnote{Note that if $k$ is odd there are no $k$-regular graphs on $n$ vertices unless $n$ is even.}). In a subsequent paper \cite{sed70b}, he extended this result to any fixed $k$. We improve these linear bounds to exponential ones.


\begin{thm}\label{thm:regular}
    For any fixed integer $k\ge 3$ there are at least $2^{\Omega(n)}$ different values of $\tau(G)$ among $k$-regular connected graphs $G$ on $n$ vertices, provided that $kn$ is even.
\end{thm}

The exponential growth rate can be easily seen to be tight since, for a fixed $k$, the number of spanning trees in a $k$-regular connected $n$-vertex graph is at most $2^{O(n)}$ (see e.g.\ \cite{mckay1983spanning,alon1990number}). Hence, $\tau(G)$ can obtain at most $2^{O(n)}$ values on such graphs.

We note that a natural approach to understanding the growth rate of the tree spectrum, already introduced in 1970 by Sedl\'a\v{c}ek \cite{sed70}, is to consider a ``dual'' problem. Namely, given $t>2$, what is the minimum number of vertices in a planar graph with exactly $t$ spanning trees? Let us denote the answer by $\alpha(t)$. 
By considering a cycle on $t$ vertices we obtain a trivial upper bound of $\alpha(t)\le t$, observed by Sedl\'a\v{c}ek \cite{sed70} already in 1970. Following a number of subsequent improvements \cite{neb,AS13,aza14,CP24a}, the current state of the art bound of $\alpha(t) \le O(\log^{3/2}(n)/\log \log n)$ is due to a recent result of Stong \cite{stong} (a stronger bound in an
appropriate sense is known for multigraphs, see \cite{chan2024equality}). A natural conjecture which remains open, raised explicitly in \cite{CKP}, is that $\alpha(t) \le O(\log t)$. 
In \cite{CKP} the authors show, relying on the Bourgain-Kontorovich \cite{BK} machinery developed towards proving Zaremba's conjecture, 
that this conjecture is true for a positive proportion of integers smaller than any large $n$. We prove a natural modular analogue of the full conjecture. 
\begin{thm}\label{thm:modular-intro}
        For any $u \in \mathbb{Z}_N$, there exists a planar graph $G$ on $O(\log N)$ vertices with $\tau(G) \equiv u \bmod N$. 
\end{thm}
In fact, our result is in several ways even stronger, see \Cref{thm:modular} for more details.

While we do not use the connection to continued fractions which was used to prove the exponential growth of $|T(\gpla{n})|$ in \cite{CKP}, we do obtain 
several interesting (simple) consequences in the opposite direction. 
Here, given a sequence of integers $a_1, \ldots, a_\ell \geq 1 $, where $\ell \geq 0$,
the corresponding \emph{continued fraction} is defined as follows:
\[ [a_1,\ldots, a_\ell] := \cfrac{1}{a_1 + \cfrac{1}{\ddots + \frac{1}{a_\ell}}}.
\]

The classical Zaremba's conjecture \cite{zaremba} asserts that there exists $C>0$ such that for any $u$ there is a coprime $t<u$ such that $\frac{t}{u}=[a_1,\ldots, a_{\ell}]$ with $a_1,\ldots,a_{\ell} \le C$. It is known that the conjecture is not true for $C= 4$ but Hensley \cite{hen} 
conjectured in 1996 that $C=2$ is enough for large enough $u$. While the conjecture is still open in general, in a remarkable paper Bourgain and Kontorovich \cite{BK} showed that it holds for $(1-o(1))n$ of the values of $u$ smaller than $n$ for any large enough $n$. 
We give a very simple proof of the following weaker result (ignoring the $a_i \in \{1,2\}$ part). 

\begin{thm}\label{thm:zaremba-weak}
    For any $N$, there are at least $N^{1/4}/2-1$ values of $u<N$ such that there exists a coprime $t<u$, an $\ell \le \frac12{\log N}$ and  $a_1,\ldots, a_\ell \in \{1,2\}$, such that
    $\frac{t}{u}=[a_1,1,\ldots, a_{\ell}, 1].$
\end{thm}

\Cref{thm:modular-intro} (or rather its strengthening, \Cref{thm:modular}) 
has as an essentially immediate consequence the following ``modular'' version of Zaremba's Conjecture. We note that a similar result is described in \cite{zaremba-survey}. 

    \begin{thm}
        Let $N$ be large enough. For any $u<N$ 
	    there exist $\ell \le O(\log N)$ and $a_1,\ldots, a_{\ell} \in \{1,2\}$ such that\footnote{
We note here that one should first compute $[a_1,1,\ldots, a_{\ell}, 1]$ over $\mathbb Q$ and only then evaluate it $\bmod N$.}
        $u\equiv [a_1,1,\ldots, a_{\ell}, 1] \bmod N.$
    \end{thm} 
\textbf{Notation.}
All our logarithms are in base two and all our graphs are simple unless otherwise specified. 

\vspace{-0.3cm}
\section{Counting spanning trees}
\vspace{-0.3cm}
Given a graph $G$ and its edge $e$ we denote by $G / e$ the (multi)-graph obtained by contracting $e$ and by $G-e$ the graph obtained by deleting $e$.
Note that $\tau(G-e)$ counts the number of spanning trees of $G$ which do not use $e$, and $\tau(G / e)$ counts the number of spanning trees using $e$. In particular, this establishes the classical recursive formula $\tau(G)=\tau(G / e)+\tau(G-e)$ for computing the number of spanning trees.

We say that a vector $\colvec{2}{t}{u}$ is \emph{$n$-planar-feasible} if there exists a planar graph $G$ on up to $n$ vertices and an edge $e \in E(G)$ such that $t=\tau(G / e)$ and $u=\tau(G-e)$. We say that $(G,e)$ is a \emph{witness} for the $n$-planar-feasibility of $\colvec{2}{t}{u}$. 
The edge $e$ is called a {\em witness edge}.
Note that if a vector is $n$-planar-feasible then it can also be realized by a planar graph with {\em exactly} $n$ vertices, by attaching leaves (which does not change the number of spanning trees).

The next simple lemma shows that in order to get good lower bounds on  $|\mathcal{T}(n)|$ it suffices to show that there are many distinct planar feasible vectors. 

\begin{lem}\label{lem:ramsey}
    If there are at least $N$ distinct $n$-planar-feasible vectors, then $|\mathcal{T}(n)| \ge \sqrt{N}.$
\end{lem}
  \begin{proof}
        Let $\colvec{2}{x_1}{y_1},\ldots, \colvec{2}{x_N}{y_N}$ be distinct $n$-planar-feasible vectors.
        Note that either there is some $y$ such that $y_i=y$ for at least $\sqrt{N}$ of these vectors, or $|\{y_1,\ldots,y_N\}| \ge \sqrt{N}$. In the latter case we are done since each $y_i$ denotes the number of spanning trees of some $G-e$ for some $n$-vertex planar graph $G$ and its edge $e$, which is also an $n$-vertex planar graph. In the former case, note that $|\{x_1+y_1,\ldots,x_N+y_N\}| \ge \sqrt{N}$ since all the (at least $\sqrt{N}$) vectors agreeing in the second coordinate must have a different sum of their coordinates (since they are different vectors). Since $x_i+y_i$ is the number of spanning trees of some $n$-vertex planar graph, we again get at least $\sqrt{N}$ distinct values taken by the spanning tree function.  
    \end{proof}
    
The following lemmas establish a way of recursively generating planar-feasible vectors. Similar considerations were also used in \cite{CKP}.

\begin{lem}\label{lem:subdivide}
    If $\colvec{2}{t}{u}$ is $n$-planar-feasible then $\colvec{2}{t+u}{u}=\begin{pmatrix}
        1 & 1\\
        0 & 1
    \end{pmatrix}\colvec{2}{t}{u}$ is $(n+1)$-planar-feasible.
\end{lem}
\begin{proof}
    Take a witness $(G,e)$ for $\colvec{2}{t}{u}$ and subdivide $e$ with a single vertex to create a new graph $G'$. Let $f,f'$ be the two newly created edges. 
    Note that $G'/f$ is isomorphic to $G$, so $\tau(G'/f) = \tau(G) = t + u$.
Also, $\tau(G' - f)$ counts the number of spanning trees of $G'$ not using $f$ (forcing us to use $f'$), which equals $\tau(G-e)=u$. Thus, $(G',f)$ is a witness for $\colvec{2}{t+u}{u}$ being $(n+1)$-planar-feasible, as desired.
\end{proof}

\begin{lem}\label{lem:triangle}
    If $\colvec{2}{t}{u}$ is $n$-planar-feasible then $\colvec{2}{2t}{t+2u}=\begin{pmatrix}
        2 & 0\\
        1 & 2
    \end{pmatrix}\colvec{2}{t}{u}$ is $(n+1)$-planar-feasible.
\end{lem}
\begin{proof}
    Take a witness $(G,e)$ for $\colvec{2}{t}{u}$, add an extra vertex $v$ and join it to both endpoints of $e$ to create a new graph $G'$. We claim that $(G',e)$ is a witness for the $(n+1)$-planar-feasibility of $\colvec{2}{2t}{t+2u}.$ Indeed, $\tau(G' / e)=2\tau(G / e)=2t$ since we have to pick one of the two new edges to connect $v$ as well as connect $G / e$. For $\tau(G' - e)$, we can either use exactly one of the two new edges giving a contribution of $2\tau(G - e)= 2u$, or both new edges giving a contribution of $\tau(G / e)=t$. Hence, $\tau(G' - e) = t + 2u$, as desired. 
\end{proof}

\begin{lem}\label{lem:triangle-2}
    If $\colvec{2}{t}{u}$ is $n$-planar-feasible then $\colvec{2}{2t+u}{t+u}=\begin{pmatrix}
        2 & 1\\
        1 & 1
    \end{pmatrix}\colvec{2}{t}{u}$ is $(n+1)$-planar-feasible.
\end{lem}
\begin{proof}
    Take a witness $(G,e)$ for $\colvec{2}{t}{u}$, add an extra vertex $v$ and join it to both endpoints of $e$ to create a new graph $G'$. Denote the two new edges by $f,f'$. We claim that $(G',f)$ is a witness for the $(n+1)$-planar-feasibility of $\colvec{2}{2t}{t+2u}.$ Indeed, $\tau(G' / f)$ is equal to the number of spanning trees of $G$ with the edge $e$ doubled (i.e., replaced with two parallel edges), which equals to 
    $2\tau(G / e)+\tau(G-e)=2t+u$. Furthermore, $\tau(G' - f)=\tau(G)=t+u,$ because every spanning tree of $G'$ not using $f$ is obtained from a spanning tree of $G$ by adding $f'$. This proves the lemma.
\end{proof}

Let us denote by 
$$A:=\begin{pmatrix}
        1 & 1\\
        0 & 1
    \end{pmatrix}
    \quad\quad\quad
    B:=\begin{pmatrix}
        2 & 0\\
        1 & 2
    \end{pmatrix}
    \quad\quad\quad    
    C:=\begin{pmatrix}
        2 & 1\\
        1 & 1
    \end{pmatrix}
    \quad\quad\quad
    D:=\begin{pmatrix}
        1 & 0\\
        1 & 1
    \end{pmatrix}$$ 

We start with a weaker result which only makes use of the operations corresponding to the matrices $A$ and $C$.
\begin{thm}\label{thm:simple}
    There are at least $2^{n}$ distinct $(n+2)$-planar-feasible vectors.
\end{thm}
\begin{proof}
    Note that $\colvec{2}{1}{0}$ is $2$-planar-feasible. This implies that for any choice of $t$ and $a_1,\ldots a_t \ge 0$, by repeated use of \Cref{lem:subdivide,lem:triangle-2}, the vector  
    \begin{equation}\label{eq:u}
    \vec{u}(a_1,\ldots,a_t):=A^{a_1}CA^{a_2}C\cdots A^{a_t}C \colvec{2}{1}{0}=A^{a_1+1}DA^{a_2+1}D\cdots A^{a_t+1}D \colvec{2}{1}{0}
    \end{equation}
    is $(n+2)$-planar-feasible for $n = a_1+a_2+\ldots+a_t+t$.
    
    We further claim that all vectors $\vec{u}=\vec{u}(a_1,\ldots,a_t)$ are distinct. 
    To see this, we need the following simple fact: for every vector $v = \colvec{2}{x}{y}$ of positive numbers, $Av = \colvec{2}{x+y}{y}$ has its first coordinate larger than the second, whereas $Dv = \colvec{2}{x}{x+y}$ has its second coordinate larger than the first. 
    Now, given $\vec{u}$ as in \eqref{eq:u}, we can determine $a_1$ by multiplying the vector by $A^{-1}$ so long as the first coordinate is larger than the second. The number of times we did this is equal to $a_1+1$, because $DA^{a_2+1}D\cdots A^{a_t+1}D \colvec{2}{1}{0}$ has its second coordinate larger than the first. We then multiply by $D^{-1}$. Repeating allows us to decode $a_2,\ldots, a_t$ in the same way, where the argument stops once the current vector equals $\colvec{2}{1}{0}$ and the number of iterations gives $t$.

    Hence, for a given $n$, the number of vectors we can obtain is the number of solutions to $a_1+a_2+\ldots+a_t+t\le n$ in non-negative integers (with variable $t$). By the usual ``balls and bins'' argument, this equals $\sum_{t=0}^n \binom{n}{t}=2^n,$ as desired.
    \end{proof}

    We note that the construction we used above is equal to the one used in \cite{CKP}. We also note that the uniqueness of vectors $\vec{u}(a_1,\ldots,a_t)$ can also be concluded from the classical fact that the matrices $A,D$ generate a free semigroup \cite{free-semigroup} (although they do not generate a free group). The following corollary is an immediate consequence of \Cref{thm:simple} and \Cref{lem:ramsey}. 

    \begin{cor}\label{cor:sqrt(2)}
        For any positive integer $n$ we have $|T(\gpla{n})| \ge 2^{n/2-1}.$ 
    \end{cor}

    We next prove a stronger bound on the number of planar-feasible vectors using the operations behind all three matrices $A,B,$ and $C$. We note that we did not choose the optimal value of the parameters in the following result for the sake of simplicity, since even after optimization the base of the exponent here is extremely unlikely to be best possible and only does slightly better than Theorem \ref{thm:simple}. The main purpose of the stronger theorem here is to showcase that one can substantially improve \Cref{thm:simple} by using additional operations and to showcase a route towards further improvement.
    
\begin{thm}\label{thm:stronger}
    There are at least
    $\lfloor \frac{1}{2\sqrt{2}} (1+\sqrt{2})^{n+1} \rfloor$
    distinct $(n+2)$-planar-feasible vectors.
\end{thm}
\begin{proof}
    Let $\mathcal{M}$ be the set of matrices of the form 
    $A^a C$ or $B^a C$ for $a \geq 0$. The weight $w(M)$ of such a matrix $M \in \mathcal{M}$ is $a+1$. For $M_1,\dots,M_t \in \mathcal{M}$, let
    \begin{equation}\label{eq:vec ABC}
    \vec{v}(M_1,\dots,M_t) := M_1 \dots M_t \colvec{2}{1}{0}
    \end{equation}
    By repeated use of \Cref{lem:subdivide,lem:triangle,lem:triangle-2}, $\vec{v}(M_1,\dots,M_t)$ is 
    $(n+2)$-planar-feasible for $n = w(M_1) + \dots + w(M_t)$.

    Next, let us show that all vectors obtained in this way are distinct. We claim that for a vector $\colvec{2}{x}{y}$ with $x > 0, y \geq 0$ and for a matrix $M \in \mathcal{M}$, knowing $\colvec{2}{u}{v} := M \colvec{2}{x}{y}$ allows us to infer $M$. This implies that $\vec{v}(M_1,\dots,M_t)$ determines $M_1,\dots,M_t$ (by first finding $M_1$, then $M_2$, etc.), and so the above vectors are distinct. 
    Note that
    $$
    A^a C \colvec{2}{x}{y} = 
    \begin{pmatrix}
        1 & a\\
        0 & 1
    \end{pmatrix}
    \begin{pmatrix}
        2 & 1\\
        1 & 1
    \end{pmatrix}
    \colvec{2}{x}{y} = 
    \colvec{2}{(2+a)x+(1+a)y}{x+y},
    $$
    and so $\colvec{2}{u}{v} = A^a C \colvec{2}{x}{y}$ satisfies $\frac{u}{v} \in (1+a,2+a]$. Similarly,
    $$
    B^a C \colvec{2}{x}{y} = 2^{a-1} \cdot 
    \begin{pmatrix}
        2 & 0\\
        a & 2
    \end{pmatrix}
    \begin{pmatrix}
        2 & 1\\
        1 & 1
    \end{pmatrix}
    \colvec{2}{x}{y} = 
    2^{a-1} \colvec{2}{4x+2y}{(2a+2)x+(a+2)y},
    $$
    and so 
    $\colvec{2}{u}{v} = B^a C \colvec{2}{x}{y}$ satisfies $\frac{u}{v} \in (\frac{2}{a+2},\frac{2}{a+1}]$. Since the intervals $(1+a,2+a]$ for $a \geq 0$ and 
    $(\frac{2}{a+2},\frac{2}{a+1}]$ for $a \geq 1$ are pairwise disjoint, knowing $\colvec{2}{u}{v} = M \colvec{2}{x}{y}$ for $M \in \mathcal{M}$ allows us to infer $M$, as claimed.

    It remains to estimate the number of sequences $(M_1,\dots,M_t)$ with $M_1,\dots,M_t \in \mathcal{M}$ and $w(M_1) + \dots + w(M_t) \leq n$. For $k \geq 0$, the number of matrices $M \in \mathcal{M}$ of weight $k$ is $1$ if $k=1$ (namely, $A^0C = B^0C$) and $2$ if $k \geq 2$ (i.e., $A^{k-1}C, B^{k-1}C$). 
    The generating function of this sequence (i.e. the sequence $a_k = 1$ for $k=1$ and $a_k = 2$ for $k \geq 2$) is $f(x) = 2\sum_{k=2}^\infty x^k + x = \frac{x(1+x)}{1-x}$.  
    Hence, by a standard fact about generating functions, the generating function for the number of sequences
    $M_1,\dots,M_t \in \mathcal{M}$ with 
    $w(M_1) + \dots + w(M_t) = n$ (i.e., total weight exactly $n$) is
    $$
    \frac{1}{1 - f(x)} = \frac{1-x}{1-2x-x^2}.
    $$
    To get the generating function for the number $m_n$ of sequences $M_1,\dots,M_t \in \mathcal{M}$ with total weight at most $n$, we need to multiply by $\frac{1}{1-x}$. Thus,
    $$
    \sum_{n=0}^\infty m_n x^n = \frac{1}{1-2x-x^2} = 
    \frac{\sqrt{2}+1}{2\sqrt{2}} \cdot \frac{1}{1 - (\sqrt{2}+1)x} + 
    \frac{\sqrt{2}-1}{2\sqrt{2}} \cdot \frac{1}{1 + (\sqrt{2}-1)x}
    .
    $$
    It follows that 
    $m_n = \frac{1}{2\sqrt{2}} \cdot \left( (1+\sqrt{2})^{n+1} - (1-\sqrt{2})^{n+1} \right) \geq 
    \lfloor \frac{1}{2\sqrt{2}} (1+\sqrt{2})^{n+1} \rfloor$, where the inequality holds because $m_n$ is an integer and
    $|(1-\sqrt{2})^{n+1}| < 1$.
    This proves the theorem. 
    \end{proof}

    We note that the construction we used above is different compared to the one used in \cite{CKP} and \Cref{thm:simple}, both of which correspond to taking vectors of the form $A^{a_1} B^2 \dots A^{a_t} B^2 \colvec{2}{1}{0}$. 
    We suspect that most of the vectors of the form 
    $A^{a_1} B^{b_1} \dots A^{a_t} B^{b_t} \colvec{2}{1}{0}$
    are unique, which would lead to a slightly better bound of about $2.61^{n}$. This is essentially the limit of what can be achieved by using matrices generated by $A$ and $B$. The issue here is that $A$ and $B$ do not generate a free semigroup \cite{free-semigroup}. Similarly, as before, \Cref{lem:ramsey} combined with \Cref{thm:stronger} gives \Cref{thm:planar}. 

    As written, our constructions do not produce regular graph witnesses. The following argument shows how we can regularize them to obtain \Cref{thm:regular}.
    \begin{proof}[ of \Cref{thm:regular}]
        We first show that we can attain exponentially many values among planar graphs with maximum degree at most $3$. To see this, suppose that we have a graph $G$ with this property and note that if $e$ is an edge of $G$ with both endpoints having degree at most two, then by adding $i\ge 2$ vertices in such a way that together with $e$ they make a $C_{i+2}$, we preserve the maximum degree property. Note also that since $i \ge 2$, one of the new edges will have both endpoints having degree equal to two, denote it by $f$ and the new graph by $G'$. Then, we have $\colvec{2}{\tau(G'/ f)}{\tau(G'-f)}=A^{i-1}C \colvec{2}{\tau(G/ e)}{\tau(G-e)}$. This can be easily seen directly or by noticing that the construction used in \Cref{lem:triangle-2} places a triangle on top of $e$ and marks one of the new edges by $f$. We can then subdivide $f$ for $i-1$ times using \Cref{lem:subdivide}, giving us the stated product. 
        Starting with an edge and repeatedly applying this operation, this implies that for every $t \ge 1$ and $a_1,\ldots, a_t \ge 1$, the vector $$\vec{w}(a_1,\ldots,a_t):=A^{a_1}CA^{a_2}C\cdots A^{a_t}C \colvec{2}{1}{0}$$
        is $(n+2)$-planar-feasible for $n = a_1+\ldots +a_t+t$ with a graph whose degrees are all equal to $2$ or $3$. As argued in the proof of \Cref{thm:simple}, all these vectors are distinct, giving at least 
        $\sum_{t=1}^{n/2-1} \binom{n-t-1}{t-1} \ge 2^{(2/3-o(1))n}$ different such vectors. 

        By the same argument as in the proof of \Cref{lem:ramsey}, this implies there are at least $2^{(1/3-o(1))n/6}$ planar graphs $G$ with distinct $\tau(G)$, each with at most $n/6$ vertices, and having all degrees equal to $2$ or $3$ (we can always remove vertices of degree one without changing the number of spanning trees). By the pigeonhole principle, there are some $0\le i \leq m \le n/6$ such that there are at least $2^{(1/3-o(1))n/6}/n^2\ge 2^{(1/3-o(1))n/6}$ such graphs with exactly $m$ vertices and exactly $i$ vertices of degree $2$. Fixing any such $G$, we now modify $G$ as follows: 
        Let $H$ denote the graph obtained from $K_4$ by subdividing one edge once, so that $H$ has exactly one vertex of degree 2, and all other vertices are of degree 3. For each of the $i$ vertices $v \in V(G)$ of degree 2, add a new copy of $H$ and connect the unique degree-2 vertex of $H$ to $v$.
%
%
        As $|V(H)| = 5$ and $i \leq m \leq n/6$, the resulting graph has exactly $m+5i \leq n$ vertices. 
        Moreover, this graph is cubic, and its number of spanning trees is equal to $\tau(G) \cdot \tau(H)^i$. Since we fixed $i$, the term $\tau(H)^i$ is independent of $G$ from our family, so we obtain at least $2^{(1/18-o(1))n}$ different values taken by $\tau(G)$ on planar cubic graphs on at most $n$ vertices. 

If $k>3$, we can in our construction first add a triangle on top of our witness edge, while keeping the same witness edge, and then add another one where we move it to a new edge and then subdivide. Our first operation is captured by multiplying by matrix $B$ (by \Cref{lem:triangle}) so each iteration corresponds to multiplication by $A^{i-1}CB$. By the same argument we still have $2^{\Omega(n/k)}$ such graphs on up to $n/(k+3)$ vertices with different values of $\tau(\cdot)$, with the key difference being that the graphs we construct all have degrees $2$ or $4$.
As before, by pigeonholing, we may assume that all these graphs have the same number of vertices $m$, and also the same number of vertices of degree $2$. 
We now proceed similarly as in the $k=3$ case, but with a different choice of ``pendant" graphs $H$. 
Namely, take $H$ to be a graph with exactly $k-2$ vertices of degree $k-1$ and the remaining vertices all of degree $k$.
To see that such an $H$ exists we may take $K_{k+1}$, remove a path of length $k-1$ and add a new vertex joined to all vertices on the removed path. Let also $H'$ be a graph obtained from $H$ by adding an edge between two vertices of degree $k-1$, so that $H'$ has $k-4$ vertices of degree $k-1$ and all others of degree $k$.
Now append copies of $H,H'$ to the vertices of $G$ as follows:
For each vertex $v$ of $G$ of degree $2$, add a copy of $H$ and join $v$ to the $k-2$ vertices of $H$ of degree $k-1$; for each vertex $v$ of $G$ of degree $4$, add a copy of $H'$ and join $v$ to the $k-4$ vertices of $H'$ of degree $k-1$. The resulting graph is $k$-regular and its number of spanning trees is $\tau(G) \cdot \tau(H)^i \cdot \tau(H')^j$, where $i$ (resp. $j$) is the number of degree-2 (resp. degree-4) vertices in $G$. As $i,j$ are fixed, all of these graphs have different $\tau(\cdot)$-values.
Also, the number of vertices of this graph is $m + m \cdot V(H) = m(k+3) \leq n$. 

The above shows that for every $k\ge 3$, we can find $2^{\Omega(n/k)}$ $k$-regular graphs on the same number of vertices $n' \leq n$ with different $\tau(\cdot)$-values. In order to get graphs on exactly $n$ vertices, we 
slightly modify our construction as follows:
Take a graph $H''$ on $k+2+n-n'$ vertices which has $k-2$ vertices of degree $k-1$ and all other vertices of degree $k$; such an $H''$ can be obtained by starting with an arbitrary $k$-regular graph on $k+1+n-n'$ vertices, removing a path of length $k-1$, and adding a vertex joined to all vertices of the path. Now, in the construction, replace one of the copies of $H$ that we add with a copy of $H''$. This brings the number of vertices to $n$, and only multiplies the tree counts of every graph in our family by the same amount, leaving them all distinct. 
%
\end{proof}

\vspace{-0.4cm}

\section{Hitting any value modulo $N$}
\vspace{-0.1cm}
   In this section we prove our modular version of the dual question asking how large a (planar) graph $G$ do we need to take in order to attain 
a given value of $\tau(G)$. The following is our stronger variant of \Cref{thm:modular-intro}.

    \begin{thm}\label{thm:modular}
        Let $N$ be sufficiently large. For any coprime $a,b$, there exists $t \le O(\log N)$ and $i_1,\ldots,i_t \in \{1,2\}$ such that $$\colvec{2}{a}{b}\equiv A^{i_1}DA^{i_2}D\ldots A^{i_t}D \colvec{2}{1}{0} \bmod N.$$ In particular, there exists a $O(\log N)$-planar-feasible vector $\colvec{2}{x}{y}$ such that $x\equiv a \bmod N$ and $y \equiv b \bmod N$. 
    \end{thm}

    Before turning to the proof, we will need a few definitions and a well known fact about the expansion in Cayley graphs of $\SL_2(\mathbb{Z}_N).$ We say that an $n$-vertex graph is a {\em $c$-expander} if any subset $U$ on up to $n/2$ vertices has its external neighborhood $N(U)$ of size at least $c|U|$. Given a group $G$ and a subset of its elements $S$, the Cayley digraph of $G$ generated by $S$, denoted $Cay(G,S)$, is the graph whose vertex-set is $G$ and where an edge from $a$ to $b$ exists if and only if $a^{-1}b \in S$. In case $S$ is symmetric (i.e., $a \in S \implies a^{-1} \in S$), we have an edge from $a$ to $b$ if and only if we have an edge from $b$ to $a$, so it is customary to drop the directions and refer to this as the Cayley graph of $G$ generated by $S$.

    Selberg's Theorem \cite{selberg1965estimation} implies that $\Cay(\SL_2(\mathbb{Z}_N),S \cup S^{-1})$ is an expander provided $S \cup S^{-1}$ generates $\SL_2(\mathbb{Z})$ (in fact whenever it generates a subgroup of finite index). It is well known that for $S=\{A,D\}$ we have that $S \cup S^{-1}$ generates $\SL_2(\mathbb{Z})$. If we set $X=AD, Y=A^2D$ we have $A=YX^{-1}$ and $D=XY^{-1}X$ so $X,Y,X^{-1},Y^{-1}$ generate $\SL_2(\mathbb{Z})$. This implies the following proposition, see also \cite{Lubotzky-survey,bourgain-gamburd} for more details on the expansion of Cayley graphs in $\SL_2(\mathbb{Z}_N)$.
    
    \begin{prop}\label{prop:sl2-expansion}
        There exists $c>0$ so that for any large enough $N$, the 
	    Cayley graph $\Cay(\SL_2(\mathbb{Z}_N),S \cup S^{-1})$, with $S=\{AD,A^{2}D\},$ is a $c$-expander.
    \end{prop}

    Another ingredient is an observation that having expansion in the Cayley graph $\Cay(\SL_2(\mathbb{Z}_N),S \cup S^{-1})$ implies expansion in the Cayley digraph $\Cay(\SL_2(\mathbb{Z}_N),S)$. Here, we say that an $n$-vertex digraph is a {\em $c$-out-expander} if any subset $U$ on up to $n/2$ vertices has its external out-neighborhood $N^+(U)$ of size at least $c|U|$. 

    \begin{prop}\label{prop:usual-to-directed-expansion}
        Let $G$ be a finite group and $S \subset G \setminus \{e\}$ (where $e$ is the identity element). If $\Cay(G,S \cup S^{-1})$ is a $c$-expander for some $c>0,$ then $\Cay(G,S)$ is a $\frac{c}{2|S|}$-out-expander.
    \end{prop}
    \begin{proof}
        Let $n=|G|$ and let $U \subseteq G$ have size $u=|U| \le n/2$. By the $c$-expansion of $H=\Cay(G,S \cup S^{-1})$, $H$ has at least $c|U|$ edges between $U$ and $G \setminus U$. Since $D=\Cay(G,S)$ is Eulerian, we have 
        \begin{align*}0=\sum_{v\in U} (d_D^+(v)-d_D^{-}(v))=&\sum_{v\in U} (d_{D}^+(v,G \setminus U)-d_D^{-}(v,G \setminus U))+\cancelto{0}{\sum_{v\in U} (d_{D[U]}^+(v)-d_{D[U]}^{-}(v))} \implies\\[2.5ex]
        &\sum_{v\in U} d_{D}^+(v,G \setminus U)=\sum_{v \in U} d_D^{-}(v,G \setminus U)=\frac12 \sum_{v \in U} d_H(v,G \setminus U) \ge \frac{c}{2} \cdot |U|.
        \end{align*}
        Since the maximum indegree in $D$ is $|S|$, we obtain that there are at least $\frac{c}{2|S|} \cdot |U|$ vertices in $N^{+}(U)$ in $D$, as desired. 
    \end{proof}

    We are now ready to prove \Cref{thm:modular}.
    
    \begin{proof}[ of \Cref{thm:modular}]
     Let $S=\{AD,A^2D\}$. By \Cref{prop:sl2-expansion} $\Cay(\SL_2(\mathbb{Z}_N),S \cup S^{-1})$ is a $c$-expander for some $c>0$. By \Cref{prop:usual-to-directed-expansion}, both $\Cay(\SL_2(\mathbb{Z}_N),S)$ and $\Cay(\SL_2(\mathbb{Z}_N),S^{-1})$ are $\frac{c}{4}$-out-expanders. Starting from an arbitrary vertex of $\Cay(\SL_2(\mathbb{Z}_N),S)$, by repeatedly expanding we reach at least $(1+c/4)^i$ vertices after $i$ steps and in particular, we reach more than half of the vertex set in at most $\log_{1+c/4}(N^3)$ steps. Similarly, starting from any vertex, by repeatedly expanding $\Cay(\SL_2(\mathbb{Z}_N),S^{-1})$ we reach more than half the vertices in at most $\log_{1+c/4}(N^3)$ steps. This shows that starting from, say, the identity matrix, we can reach any other matrix in $\SL_2(\mathbb{Z}_N)$ in at most  $2\log_{1+c/4}(N^3) \le O(\log N)$ multiplications by elements of $S$. 
    
    Now, given coprime $a,b$, we choose $e,f \in \mathbb{Z}$ such that $ae-bf=1$, so that 
    $M := \begin{pmatrix}
        a & e\\
        b & f\\
    \end{pmatrix} \in \SL_2(\mathbb{Z}_N)$. By the above, we can write $M$ as the product of $O(\log N)$ elements of $S$. The elements of $S$, i.e. $AD = C$ and $A^2D = AC$, both correspond to graph operations (via \Cref{lem:subdivide,lem:triangle-2}). Hence, $M \colvec{2}{1}{0}$ is a $O(\log N)$-feasible vector. 
    Also, $M \colvec{2}{1}{0}=\colvec{2}{a}{b}$ by the choice of $M$, as required. 
    \end{proof}

    We note that \Cref{thm:modular} gives another proof of the exponential growth of the tree spectrum function. 

\vspace{-0.3cm}    
\section{An application to continued
fractions}
\vspace{-0.1cm}
The connection to the theory of continued fractions goes via the following, well-known association between continued fractions and matrix products. We note that the second equivalence was observed in  \cite{CKP} and involves precisely the matrices $A$ and $D$ we encounterd in the previous sections. Given $0\le t<u$ such that $\gcd(t,u)=1$ we get: 
\begin{align}
     \frac{t}{u}=[a_1,b_1,\ldots, a_{\ell}, b_\ell] \quad &\Leftrightarrow \quad 
\colvec{2}{u}{t}
=\begin{pmatrix}
  {a_1} & 1 \\
  1 & 0
\end{pmatrix}
\begin{pmatrix}
  {b_1} & 1 \\
  1 & 0
\end{pmatrix}
 \cdots
\begin{pmatrix}
  {a_\ell} & 1 \\
  1 & 0
\end{pmatrix}
\begin{pmatrix}
  {b_\ell} & 1 \\
  1 & 0
\end{pmatrix}\colvec{2}{1}{0}\nonumber \\
\quad &\Leftrightarrow \quad 
\colvec{2}{u}{t}
=\begin{pmatrix}
  1 & {a_1} \\
  0 & 1
\end{pmatrix}
\begin{pmatrix}
  1 & 0 \\
  {b_1} & 1
\end{pmatrix}
 \cdots
\begin{pmatrix}
  1 & {a_\ell} \\
  0 & 1
\end{pmatrix}
\begin{pmatrix}
  1 & 0 \\
  {b_\ell} & 1
\end{pmatrix}\colvec{2}{1}{0}\nonumber \\
\quad &\Leftrightarrow \quad 
\colvec{2}{u}{t}
=A^{a_1}D^{b_1}\cdots A^{a_\ell}D^{b_\ell}\colvec{2}{1}{0}. \label{eq:continued-fractions}
\end{align}



We are now ready to prove \Cref{thm:zaremba-weak}.

\begin{proof}[ of \Cref{thm:zaremba-weak}]
    Let $n= \floor{\frac{\log N}{2}}$. 
    As argued in the proof of \Cref{thm:simple}, if we pick $\ell \le n$ and $1\le a_2,\ldots a_\ell \le 2$ the vectors $$\vec{v}(a_2,\ldots,a_\ell):=ADA^{a_2}D\cdots A^{a_\ell}D \colvec{2}{1}{0}$$
    are all distinct and there are at least $1+2+\ldots + 2^{n-1}=2^{n}-1$ of them. Note that if $\vec{v}(a_2,\ldots,a_\ell)= \colvec{2}{u}{t},$ then $t<u \le 4^{n} \le N,$ since the maximum entry of $(A^2D)^{n}$ is at most $4^n$. Note further that the ratio $\frac{t}{u}$ has the required continued fraction expansion by \eqref{eq:continued-fractions}. In addition, we know that $\gcd(t,u)=1$ since multiplying a vector with coprime coordinates by either $A$ or $D$ preserves this property (since $\gcd(x+y,y)=\gcd(x,y)=\gcd(x,x+y)$).

    Now, by the same argument behind \Cref{lem:ramsey}, we conclude that either at least $2^{n/2}-1$ of the vectors have distinct first coordinates or at least $2^{n/2}-1$ of them have distinct sums. In the former case we are done immediately, and in the latter we may multiply our vectors by $A$ to obtain a vector with first coordinate being the sum, which are in particular then all distinct and still smaller than $N$. In either case we get at least $2^{n/2}-1\ge N^{1/4}/2-1$ distinct $u<N$ for which there is a coprime $t$ with 
the desired properties.
\end{proof}

We note that one can easily improve the exponent in \Cref{thm:zaremba-weak}, for example by restricting attention
in the above proof to only vectors which have precisely $n/2$ exponents equal to $1/2$. It can also be improved a bit further by relaxing either the restriction $a_i \in \{1,2\}$ to instead allow the $a_i$'s to be smaller than a fixed constant $C$, or by removing the requirement that the second entry is one. However, the limit of our approach here is always $N^{1/2}$ due to the use of the argument behind \Cref{lem:ramsey}. 





\textbf{Acknowledgements.}
We thank Swee Hong Chan, Alex Kontorovich and Igor Pak for helpful 
comments on an early version of the paper, and 
thank Boris Bukh, Stefan Glock and Leo Versteegen for useful discussions. 
We are also indebted to an anonymous referee for suggesting an improved proof for Theorem \ref{thm:stronger}, yielding a better bound.



\providecommand{\MR}[1]{}
\providecommand{\MRhref}[2]{%
  \href{http://www.ams.org/mathscinet-getitem?mr=#1}{#2}
}

   \bibliographystyle{amsplain_initials_nobysame}
   \bibliography{ref}

\end{document}